\documentclass[12pt,reqno]{amsart}
\usepackage{enumitem}
\usepackage{hyperref}
\usepackage{amsrefs}

\newtheorem{theorem}{Theorem}[section]
\newtheorem{corollary}[theorem]{Corollary}
\newtheorem{definition}[theorem]{Definition}
\newtheorem{lemma}[theorem]{Lemma}
\newtheorem{remark}[theorem]{Remark}

\def\1#1{_{\ell^\infty(#1)}}\def\2#1{_{\ell^\infty(1/#1)}}

\begin{document}
\title[End-point Norm Estimates]{End-point Norm Estimates for Ces\`aro and Copson Operators}
\author{Sorina Barza}
\author{Bizuneh Minda Demissie}
\author{Gord Sinnamon}
\address{Department of Mathematics and Computer Science, Karlstad University, SE-65188 Karlstad, SWEDEN 
{\em E-mail: sorina.barza@kau.se} }
\address{Department of Mathematics and Computer Science, Karlstad University, SE-65188 Karlstad, SWEDEN and Department of Mathematics, Addis Ababa University, 1176 Addis Ababa, Ethiopia\newline
 {\em E-mail: bizuneh.minda@aau.edu.et} }
\address{Department of mathematics, University of Western Ontario, London, Canada 
{\em E-mail: sinnamon@uwo.ca}}


\begin{abstract}
For a large class of operators acting between weighted $\ell^\infty$ spaces, exact formulas are given for their norms and the norms of their restrictions to the cones of nonnegative sequences; nonnegative, nonincreasing sequences; and nonnegative, nondecreasing sequences. The weights involved are arbitrary nonnegative sequences and may differ in the domain and codomain spaces. The results are applied to the Ces\`aro and Copson operators, giving their norms and their distances to the identity operator on the whole space and on the cones. Simplifications of these formulas are derived in the case of these operators acting on power-weighted $\ell^\infty$. As an application, best constants are given for inequalities relating the weighted $\ell^\infty$ norms of the Ces\`aro and Copson operators both for general weights and for power weights.
\end{abstract}
\maketitle

\section{Introduction}
 
The Ces\`aro matrix, $C$, and its transpose the Copson matrix, $C^*$, are
\[
C={\scriptstyle\left(\begin{smallmatrix}
1&&&&\\
\vphantom{\frac12}\frac12&\frac12&&&\\
\vphantom{\frac12}\frac13&\frac13&\frac13&&\\
\vphantom{\frac12}\frac14&\frac14&\frac14&\frac14&\\
\vphantom{\frac12}\vdots&\vdots&\vdots&\vdots&\ddots\\
&\hphantom{-1}&\hphantom{-1}&\hphantom{-1}&\hphantom{-1}
\end{smallmatrix}\right)}
\quad\mbox{and}\quad 
C^*={\scriptstyle\left(\begin{smallmatrix}  
\vphantom{\frac12}1&\frac12&\frac13&\frac14&\dots\\
\vphantom{\frac12}&\frac12&\frac13&\frac14&\dots\\
\vphantom{\frac12}&&\frac13&\frac14&\dots\\
\vphantom{\frac12}&&&\frac14&\dots\\
&&&&\ddots\\
&\hphantom{-1}&\hphantom{-1}&\hphantom{-1}&\hphantom{-1}
\end{smallmatrix}\right)}.
\]
The same names denote the operators associated with these infinite matrices, defined by
\[
(Cx)_n=\frac1n\sum_{k=1}^nx_k \quad\mbox{and}\quad
(C^*x)_n=\sum_{k=n}^\infty\frac{x_k}k
\]
for $x$ an appropriate real sequence.  The motivation for this work is to determine best constants in the weighted two-operator inequalities
\begin{equation}\label{CandC*}
\|Cx\|\1v\le A\|C^*x\|\2u\quad \mbox{and}\quad 
 \|C^*x\|\1v\le A\|Cx\|\2u
\end{equation}
for all $x$ and also for all nonnegative $x$.

The Ces\`aro and Copson operators, together with their integral analogues,
\[
Hf(x)=\frac1x\int_0^xf(t)\,dt\quad\text{and}\quad H^*f(x)=\int_x^\infty f(t)\frac{dt}t,
\]
appear throughout classical and modern analysis. They were already standard tools in Fourier analysis when Hardy used them to give a simple proof of Hilbert's double series theorem from complex analysis. They serve as base cases and motivating examples for summability, positive operators, convolution inequalities, interpolation of operators, maximal functions and more.

Most relevant to our study, is their appearance in the theory of weighted norm inequalities. A remarkable array of techniques have been tried out on these operators for the first time and often the results set the standard for subsequent progress.


Recently, techniques for determination of exact operator norms, exact distances between operators, and best constants in two-operator inequalities have been worked out using $H$, $H^*$, $C$ and $C^*$ as motivating examples.

For the operators $H$ and $H^*$ a great deal of progress has been made in recent years. We refer to \cite{S17} and \cite{S20}, which, besides establishing the current best results for exact operator norms, include in their introductions detailed accounts of recent work. The contributions of Boza and Soria deserve special mention, recently from \cite{BS19} and \cite{BS20}, but going back to \cite{BS11}. In the first, they make a clear case for the independent study of restrictions of operators to cones of monotone functions. In the second, they point out the significance of understanding the action of operators in endpoint cases, i.e., the $p=1$,  $p=\infty$, and weak type cases among $\ell^p$ spaces.

For the operators $C$ and $C^*$, exact norms, distances and constants had already found an important place in Bennett's 1996 memoir \cite{B}. Some were proved and others were left as open problems. A few of the open problems have settled quite recently, see \cite{J, K, Si}.

Our focus on weighted $\ell^\infty$ spaces puts us firmly in the endpoint case, and greatly simplifies norm estimates. On the other hand our results apply for general weight sequences, something which is beyond the current reach when seeking exact operator norms in the $\ell^p$ spaces for $1<p<\infty$. We also consider the restrictions of operators to cones of monotone sequences, something of proven value.

Our approach is in two steps. First, we reduce the best constant problems for the two-operator inequalities \eqref{CandC*} for general $x$ or for nonnegative $x$ to the determination of the operator norm of a related matrix operator on a related cone of sequences. See Theorems \ref{CC*reduction} and \ref{C*Creduction}. Second, we prove and apply a result on matrix operator norms between cones in weighted $\ell^\infty$ spaces that is general enough to include the ones we need to solve the best constant problems. See Theorem \ref{NormOfB}. This result is of independent interest and we apply it to give the operator norms of a number of related matrix operators that have appeared in recent literature. Here the operators $C-I$ and $C^*-I$ figure prominently. The results of our analysis of \eqref{CandC*} are in Theorem \ref{BestConstants}.

The most commonly studied and applied weight sequences are the power weights. We illustrate our results throughout by giving concrete expressions for the best constants in the case of power weighted $\ell^\infty$. See Theorems \ref{SpecCpower} and \ref{SpecC*power} for exact operator norms for $C$ and $C^*$ on all four cones. See \ref{SpecCIpower} for exact distances from $C$ to the identity on all four cones. The exact distance from $C^*$ to the identity is given in \ref{SpecC*Ipower} on two of the cones. (The case of nondecreasing sequences is trivial and the case of nonincreasing sequences remains open.) The best constants in the two-operator inequalities are given in  \ref{powerCC*} and \ref{powerC*C} on the cone of all sequences and on the cone of all nonnegative sequences. 

\subsection{Notation and Definitions} For an infinite matrix to represent an operator on sequences, we have to decide in what sense the sums involved in matrix multiplication should converge.
\begin{definition} Let $B=(b_{n,k})_{n\ge1,k\ge1}$ be a real matrix. The domain, denoted $\mathcal D(B)$, of the associated matrix operator is the set of all real sequences $x$ such that for each $n$, the sum $\sum_{k=1}^\infty b_{n,k}x_k$ converges to a real number. For $x\in\mathcal D(B)$, we define the sequence $Bx$ by setting $(Bx)_n=\sum_{k=1}^\infty b_{n,k}x_k$.

If all entries of $B$ and $x$ are nonnegative, we use $(Bx)_n$ to denote the above sum even when $x\notin \mathcal D(B)$ by permitting $(Bx)_n$ to take the value $\infty$.
\end{definition}

This definition gives us larger domains than if we insisted on absolute convergence in all matrix sums. It means that our matrix operators don't correspond to standard integral operators as well as they correspond to operators defined by principal value integrals.


Besides $C$ and $C^*$ we will encounter the matrices $I$, $S$, $S^*$, $D$ and $E$. The first three are standard, the identity matrix, the right shift (with ones on the subdiagonal) and the left shift (with ones on the superdiagonal.) The other two are defined by 
$$
D={\scriptstyle\left(\begin{smallmatrix}  
\vphantom{\frac12}\frac12&&&&\\
\vphantom{\frac12}&\frac13&&&\\
\vphantom{\frac12}&&\frac14&&\\
\vphantom{\frac12}&&&\frac15&\\
&&&&\ddots\\
\hphantom{-1}&\hphantom{-1}&\hphantom{-1}&\hphantom{-1}&\hphantom{-1}
\end{smallmatrix}\right)}
\quad\mbox{and}\quad
E={\scriptstyle\left(\begin{smallmatrix}
\vphantom{\frac12}1&&&&\\
\vphantom{\frac12}1&1&&&\\
\vphantom{\frac12}1&1&1&&\\
\vphantom{\frac12}1&1&1&1&\\
\vdots&\vdots&\vdots&\vdots&\ddots\\
\hphantom{-1}&\hphantom{-1}&\hphantom{-1}&\hphantom{-1}&\hphantom{-1}
\end{smallmatrix}\right)}
$$
so that $(Dx)_n=x_n/(n+1)$ and $(Ex)_n=\sum_{k=1}^nx_k$. The domain of each these matrix operators, with the exception of $C^*$, consists of all real sequences. Evidently, $\mathcal D(C^*)$ consists of all real sequences $x$ for which $\sum_{k=1}^\infty\frac{x_k}k$ converges in $\mathbb R$.

Let $\ell$, $\ell^+$, $\ell^\downarrow$ and $\ell^\uparrow$ denote, respectively, the set of all sequences of real numbers, the set of all sequences of nonnegative real numbers, the set of all nonincreasing sequences of nonnegative real numbers, and the set of all nondecreasing sequences of nonnegative real numbers. Inequalities between sequences are termwise so for $x,y\in\ell$, $x\le y$ means $y-x\in\ell^+$, that is, $x_k\le y_k$ for all $k$.


For weights $u,v\in\ell^+$ and for any $x,y\in\ell$ we define
$$
\|y\|\1v=\sup_n |y_n|v_n\quad\mbox{and}\quad\|x\|\2u=\sup_k |x_k|/u_k.
$$
The two definitions agree, except that if $u_k =0$ for some $k$, the sequence $(1/u_1,1/u_2,\dots)$ is not in $\ell^+$. In this case we apply the convention $0/0=0$: If $u_k=0$ for some $k$ then $|x_k|/u_k=\infty$ when $x_k\ne 0$ and $|x_k|/u_k=0$ when $x_k=0$. Note that we permit these weighted ``norms'' to take the value $\infty$.

For a real number $x$, let $x^+=(|x|+x)/2\ge0$ and $x^-=(|x|-x)/2\ge0$. Note that $x=x^+-x^-$. This notation extends termwise to sequences and entrywise to matrices: If $x=(x_n)$ is a real sequence, then $x^+=((x_n)^+)$, $x^-=((x_n)^-)$, and $|x|=(|x_n|)$. If $B=(b_{n,k})$ then $B^+=((b_{n,k})^+)$, $B^-=((b_{n,k})^-)$ and $|B|=(|b_{n,k}|)$.
 
For $u\in\ell^+$ we define the greatest nonincreasing minorant $u^\downarrow$ of $u$ and the greatest nondecreasing minorant $u^\uparrow$ of $u$ by
$$
u^\downarrow_k=(u^\downarrow)_k=\min_{j\le k}u_j\quad\mbox{and}\quad u^\uparrow_k=(u^\uparrow)_k=\inf_{j\ge k}u_j.
$$
Their relevance emerges from the following simple observation.
\begin{lemma}\label{downandup}  Let $u\in\ell^+$. 
\begin{enumerate}[label=(\roman*)]
\item 
If $x\in\ell^\downarrow$, then $\|x\|\2u=\|x\|\2{u^\downarrow}$;
\item If $x\in\ell^\uparrow$, then $\|x\|\2u=\|x\|\2{u^\uparrow}$;
\item If $x\in\ell^\downarrow$ and $x\le u$ then $x\le u^\downarrow$.
\item If $x\in\ell^\uparrow$ and $x\le u$ then $x\le u^\uparrow$.
\end{enumerate}
\end{lemma}
\begin{proof} First observe that, since $u^\downarrow\le u$ and $u^\uparrow\le u $, 
\[
\|x\|\2u\le\|x\|\2{u^\downarrow}\quad\mbox{and}\quad\|x\|\2u\le\|x\|\2{u^\uparrow}.
\]
Let $x\in\ell^\downarrow$. For each $k$, 
\[
\frac{x_k}{u^\downarrow_k}=\max_{j\le k} \frac{x_k}{u_j}\le\max_{j\le k} \frac{x_j}{u_j}\le \|x\|\2u.
\]
Take the supremum over all $k$ to get $\|x\|\2{u^\downarrow}\le\|x\|\2u$.

Let $x\in\ell^\uparrow$. For each $k$, 
\[
\frac {x_k}{u^\uparrow_k}=\sup_{j\ge k} \frac{x_k}{u_j}\le\sup_{j\ge k} \frac{x_j}{u_j}\le \|x\|\2u.
\]
Take the supremum over all $k$ to get $\|x\|\2{u^\uparrow}\le\|x\|\2u$. 

If $x\in\ell^\downarrow$ and $x\le u$, then $\|x\|\2u\le1$ so $\|x\|\2{u^\downarrow}\le1$ and therefore $x\le u^\downarrow$. If $x\in\ell^\uparrow$ and $x\le u$, then $\|x\|\2u\le1$ so $\|x\|\2{u^\uparrow}\le1$ and therefore $x\le u^\uparrow$.
\end{proof}

\section{Two Identities}

In this section we use two matrix identities to connect the inequalities \eqref{CandC*} to norm inequalities for related operators. The identities are
$$
C=(C-S^*)C^*\quad\mbox{and}\quad C^*=(C^*-S)DE.
$$
In Section 10 of \cite{B}, Bennett uses the first identity and one closely related to the second, namely, $C^*=(C^*-I)SC$, to explore two-operator inequalities involving $C$ and $C^*$. Either of the two second identities would suffice in this analysis; our aim was to simplify intermediate results.

In matrix form, the identity $(C-S^*)C^*=C$ may be written as
$$
{\scriptstyle\left(\begin{smallmatrix}  
\vphantom{\frac12}1&-1&&&&\\
\vphantom{\frac12}\frac12&\frac12&-1&&&\\
\vphantom{\frac12}\frac13&\frac13&\frac13&-1&&\\
\vphantom{\frac12}\frac14&\frac14&\frac14&\frac14&-1&\\
\vdots&\vdots&\vdots&\vdots&\vdots&\ddots\\
&\hphantom{-1}&\hphantom{-1}&\hphantom{-1}&\hphantom{-1}
\end{smallmatrix}\right)}
{\scriptstyle\left(\begin{smallmatrix}
\vphantom{\frac12}1&\frac12&\frac13&\frac14&\dots\\
\vphantom{\frac12}&\frac12&\frac13&\frac14&\dots\\
\vphantom{\frac12}&&\frac13&\frac14&\dots\\
\vphantom{\frac12}&&&\frac14&\dots\\
&&&&\ddots\\
&\hphantom{-1}&\hphantom{-1}&\hphantom{-1}&\hphantom{-1}
\end{smallmatrix}\right)}
={\scriptstyle\left(\begin{smallmatrix}
1&&&&\\
\vphantom{\frac12}\frac12&\frac12&&&\\
\vphantom{\frac12}\frac13&\frac13&\frac13&&\\
\vphantom{\frac12}\frac14&\frac14&\frac14&\frac14&\\
\vphantom{\frac12}\vdots&\vdots&\vdots&\vdots&\ddots\\
&\hphantom{-1}&\hphantom{-1}&\hphantom{-1}&\hphantom{-1}
\end{smallmatrix}\right)}.
$$
Viewed as an operator identity we can prove that it is valid on $\mathcal D(C^*)$, the domain of the operator $C^*$.

\begin{lemma}\label{firstId} If $x\in \mathcal D(C^*)$ then $(C-S^*)C^*x=Cx$.
\end{lemma}
\begin{proof} Let $x\in \mathcal D(C^*)$ and set $y=C^*x$ to see that $(Cx)_n$ is equal to
$$
\frac1n\sum_{j=1}^n\frac{x_j}j\sum_{k=1}^j1=\frac1n\sum_{k=1}^n\sum_{j=k}^n\frac{x_j}j=\frac1n\sum_{k=1}^n(y_k -y_{n+1})=(Cy)_n-y_{n+1}.
$$
Therefore $(Cx)_n=((C-S^*)y)_n=((C-S^*)C^*x)_n$ for all $n$.
\end{proof}

The second identity is a bit more complicated because the matrix multiplication involves infinite sums and extra care has to be taken with the domain of the matrix operators.

In matrix form, the identity $(C^*-S)DE=C^*$ may be written as
$$
{\scriptstyle\left(\begin{smallmatrix}  
\vphantom{\frac12}1&\frac12&\frac13&\frac14&\dots\\
\vphantom{\frac12}-1&\frac12&\frac13&\frac14&\dots\\
\vphantom{\frac12}&-1&\frac13&\frac14&\dots\\
\vphantom{\frac12}&&-1&\frac14&\dots\\
\vphantom{\frac12}&&&\ddots&\ddots\\
&\hphantom{-1}&\hphantom{-1}&\hphantom{-1}&\hphantom{-1}
\end{smallmatrix}\right)}
{\scriptstyle\left(\begin{smallmatrix}  
\vphantom{\frac12}\frac12&&&&\\
\vphantom{\frac12}\frac13&\frac13&&&\\
\vphantom{\frac12}\frac14&\frac14&\frac14&&\\
\vphantom{\frac12}\frac15&\frac15&\frac15&\frac15&\\
\vdots&\vdots&\vdots&\vdots&\ddots\\
&\hphantom{-1}&\hphantom{-1}&\hphantom{-1}&\hphantom{-1}
\end{smallmatrix}\right)}
={\scriptstyle\left(\begin{smallmatrix}  
\vphantom{\frac12}1&\frac12&\frac13&\frac14&\dots\\
\vphantom{\frac12}&\frac12&\frac13&\frac14&\dots\\
\vphantom{\frac12}&&\frac13&\frac14&\dots\\
\vphantom{\frac12}&&&\frac14&\dots\\
&&&&\ddots\\
&\hphantom{-1}&\hphantom{-1}&\hphantom{-1}&\hphantom{-1}
\end{smallmatrix}\right)}.
$$
Next we show that it is also an operator identity on $\mathcal D(C^*)$.
\begin{lemma}\label{secondId} For $x\in\ell$, $x\in\mathcal D(C^*)$ if and only if 
\[
Ex\in\mathcal D(C^*D)\quad\text{and}\quad(DEx)_N\to0\text{ as }N\to\infty.
\]
In this case $(C^*-S)DEx=C^*x$.
\end{lemma}
\begin{proof} Fix a real sequence $x$ and a positive integer $n$. If $N>n$, then
$$
\sum_{j=1}^N x_j\Big(\frac1{N+1}+\sum_{k=\max(n,j)}^N\frac1{k(k+1)}\Big)
=(DEx)_N+\sum_{k=n}^N\frac1{k(k+1)}\sum_{j=1}^kx_j,
$$
which telescopes to 
\begin{equation}\label{2C*}
\sum_{j=n}^N\frac{x_j}j+\frac1n\sum_{j=1}^{n-1}x_j
=\sum_{j=1}^N\frac{x_j}{\max(n,j)}=(DEx)_N+\sum_{k=n}^N\frac{(DEx)_k}k.
\end{equation}
Suppose $Ex\in\mathcal D(C^*D)$ and $(DEx)_N\to0$ as $N\to\infty$. Then the right-hand side of \eqref{2C*} converges as $N\to\infty$. So does the left-hand side, so $x\in\mathcal D(C^*)$.

Conversely, suppose $x\in\mathcal D(C^*)$ and let $N\to\infty$ in \eqref{2C*}. Since the left-hand side converges, so does the right-hand side. Setting $y=C^*x$ we get $y_N\to0$. It follows that the averages $(Cy)_N\to0$ and the shifts $(S^*y)_N\to0$. Now Lemma \ref{firstId} shows that $(Cx)_N=(Cy)_N-(S^*y)_N\to0$. But $(DEx)_N=\frac N{N+1}(Cx)_N$ so $(DEx)_N\to0$. Since the first term of the right-hand side of \eqref{2C*} converges, so does the second term. It follows that $Ex\in\mathcal D(C^*D)$, which completes the equivalence. 

Letting $N\to\infty$, \eqref{2C*} becomes $(C^*x)_n+(SDEx)_n=(C^*DEx)_n$. Since $n$ was arbitrary, $(C^*-S)DEx=C^*x$.
\end{proof}

These two identities are the keys to proving the following two theorems that reduce inequalities relating $C$ and $C^*$ to inequalities involving a single operator.

\begin{theorem}\label{CC*reduction} Let $u,v\in\ell^+$ and $A\in[0,\infty)$. Then \eqref{CC*} if and only if \eqref{C-S*0}, and \eqref{CC*+} if and only if \eqref{C-S*d0}, where
\begin{align}
\|Cx\|\1v\le A\|C^*x\|\2u\quad&\mbox{for } x\in\mathcal D(C^*);\label{CC*}\\
\|(C-S^*)y\|\1v\le A\|y\|\2u\quad&\mbox{for } y\in\ell, y_n\to0;\label{C-S*0}\\
\|Cx\|\1v\le A\|C^*x\|\2u\quad&\mbox{for } x\in\ell^+\cap \mathcal D(C^*);\label{CC*+}\\
\|(C-S^*)y\|\1v\le A\|y\|\2u\quad&\mbox{for } y\in\ell^\downarrow, y_n\to0.\label{C-S*d0}
\end{align}
\end{theorem}

\begin{proof} Let $x\in\mathcal D(C^*)$, set $y=C^*x$, and note that $y_n\to0$ as $n\to\infty$. If \eqref{C-S*0} holds, then by Lemma \ref{firstId},
\[
\|Cx\|\1v=\|(C-S^*)y\|\1v\le A\|y\|\2u=A\|C^*x\|\2u,
\]
so \eqref{CC*} holds. If \eqref{C-S*d0} holds and $x\in\ell^+\cap \mathcal D(C^*)$, then $y\in\ell^\downarrow$ so the same estimate gives \eqref{CC*+}.

Now let $y\in\ell$ with $y_n\to0$ and set $x_k=k(y_k-y_{k+1})$. We have
\[
\sum_{k=n}^N\frac{x_k}k=\sum_{k=n}^N(y_k-y_{k+1})=y_n-y_N\to y_n
\]
as $N\to\infty$. So $x\in\mathcal D(C^*)$ and $C^*x=y$. If \eqref{CC*} holds, then Lemma \ref{firstId} shows
\[
\|(C-S^*)y\|\1v=\|Cx\|\1v\le A\|C^*x\|\2u=A\|y\|\2u,
\]
so \eqref{C-S*0} holds. If \eqref{CC*+} holds, and $y\in\ell^\downarrow$ with $y_n\to0$, then $x\in\ell^+$ and the same estimate gives \eqref{C-S*d0}. 
\end{proof}

\begin{theorem}\label{C*Creduction} Let $u,v\in\ell^+$ and $A\in[0,\infty)$. Set $w_k=ku_k$ for each $k$. Then \eqref{C*C} if and only if \eqref{C*-S0}, and \eqref{C*C+} if and only if \eqref{C*-Su0}, where
\begin{align}
\|C^*x\|\1v\le A\|Cx\|\2u,\quad& x\in \mathcal D(C^*);\label{C*C}\\
\|(C^*-S)Dz\|\1v\le A\|z\|\2w,\quad&z\in \mathcal D(C^*D), (Dz)_n\to0;\label{C*-S0}\\
\|C^*x\|\1v\le A\|Cx\|\2u,\quad&x\in \ell^+\cap\mathcal D(C^*);\label{C*C+}\\
\|(C^*-S)Dz\|\1v\le A\|z\|\2w,\quad& z\in \ell^\uparrow\cap\mathcal D(C^*D), (Dz)_n\to0.\label{C*-Su0}
\end{align}
\end{theorem}

Proof. Let $x\in\mathcal D(C^*)$ and set $z=Ex$.  The definition of $E$ shows $\|z\|\2w=\|Cx\|\2u$. From Lemma \ref{secondId} we get $(Dz)_n\to0$, $z\in\mathcal D(C^*D)$ and $C^*x=(C^*-S)Dz$. If \eqref{C*-S0} holds, then
$$
\|C^*x\|\1v=\|(C^*-S)Dz\|\1v\le A\|z\|\2w=A\|Cx\|\2u,
$$
that is, \eqref{C*C} holds. If $x\in \ell^+$ then $z\in\ell^\uparrow$ so the same estimate shows that if \eqref{C*-Su0} holds, so does \eqref{C*C+}.

Now let $z\in\mathcal D(C^*D)$ with $(Dz)_N\to0$ as $N\to\infty$, set $z_0=0$ and $x_k=z_k-z_{k-1}$ for all $k$. Then $z=Ex$ so Lemma \ref{secondId} shows that $x\in\mathcal D(C^*)$ and $(C^*-S)Dz=C^*x$. If \eqref{C*C} holds, then 
$$
\|(C^*-S)Dz\|\1v=\|C^*x\|\1v\le A\|Cx\|\2u=A\|z\|\2w,
$$
so \eqref{C*-S0} holds. If \eqref{C*C+} holds and $z\in\ell^\uparrow$, then $x\in\ell^+$ so the same estimate gives \eqref{C*-Su0}. This completes the proof.

\section{Operator norms for some matrix operators on cones} 
 
The simple form of weighted $\ell^\infty$ norms permits direct computation of the norms of matrix operators from one weighted space to another and from the positive cone of one weighted space to another. For the cones of decreasing sequences and increasing sequences, the situation is more delicate but for each of these cones we identify a class of matrix operators for which it simplifies nicely. The operators involved in our analysis of the inequalities in \eqref{CandC*} are in those classes.

\begin{definition} Let $b\in \ell$. We say that $b$ has \emph{positives before negatives} provided that for all $j, k\in\mathbb Z^+$, $b_j>0>b_k$ only if $j<k$. We say that $b$ has \emph{negatives before positives} if $-b$ has positives before negatives.
 
Note that if $b$ has positives before negatives or has negatives before positives then $\sum_{k=1}^Nb_k$ is a monotone function of $N$ for sufficiently large $N$ so the sum $\sum_{k=1}^\infty b_k$, exists in $[-\infty,\infty]$. We call it the \emph{sum of $b$}.
\end{definition}
Let $u,v\in\ell^+$. For a matrix $B$, let $A(B)$, $A^+(B)$, $A^\downarrow(B)$ and $A^\uparrow(B)$ denote the smallest constant $A\in[0,\infty]$ such that inequality
\begin{equation}\label{B}
\|Bx\|\1{v}\le A\|x\|\2u
\end{equation}
holds for all $x\in\mathcal D(B)$, $x\in\ell^+\cap\mathcal D(B)$, $x\in\ell^\downarrow\cap\mathcal D(B)$, and $x\in\ell^\uparrow\cap\mathcal D(B)$, respectively.

\begin{remark}\label{preproc} Multiplying the matrix $B$ on the left by a complex diagonal matrix has no effect on the left-hand side of \eqref{B}, provided the weight sequence $v$ is adjusted appropriately. This simple observation substantially extends the applicability of the next theorem. Rather than unduly complicate its statement, we trust that, in applications, suitable row-by-row ``preprocessing'' will have been carried out to ensure that the hypotheses of the theorem are satisfied. One simple form of this preprocessing allows some subset of the rows of a real matrix $B$ to be multiplied by $-1$ to permit the use of parts \ref{p3} or \ref{p4} of the theorem.
\end{remark}

Some expressions in what follows need to be understood according to the convention $\infty\cdot0=0$.
\begin{theorem}\label{NormOfB} Let $B$ be a matrix with real entries.
\begin{enumerate}[label=(\roman*)]
\item\label{p1} The least $A\in [0,\infty]$ such that \eqref{B} holds
 for all $x\in\mathcal D(B)$ is 
 \[
 A(B)=\||B|u\|\1v.
 \]
\item\label{p2} The least $A\in [0,\infty]$ such that \eqref{B} holds
 for all 
$ x\in\ell^+\cap\mathcal D(B)$ is 
\[
A^+(B)=\max(\|B^+u\|\1v,\|B^-u\|\1v).
\]
\item\label{p3} Suppose each row of $B$ has positives before negatives and a nonnegative sum. The least $A\in [0,\infty]$ such that \eqref{B} holds
 for all $x\in\ell^\downarrow\cap\mathcal D(B)$ is 
 \[
 A^\downarrow(B)=\|B^+(u^\downarrow)\|\1v.
\]
\item\label{p4} Suppose each row of $B$ has negatives before positives and a nonnegative sum. If each row of $B$ has a finite sum then the least $A\in [0,\infty]$ such that \eqref{B} holds for all $x\in\ell^\uparrow\cap\mathcal D(B)$ is 
\[
A^\uparrow(B)=\|B^+(u^\uparrow)\|\1v.
\]
If some row of $B$ has an infinite sum, then $\ell^\uparrow\cap\mathcal D(B)=\{0\}$ so \eqref{B} holds trivially with $A^\uparrow(B)=0$.
\end{enumerate}
\end{theorem}
\begin{proof} Let $\rho^+$ be the result of replacing all nonzero entries of $B^+$ by $1$ and let $\rho^-$ be the result of replacing all nonzero entries of $B^-$ by $1$. Then  $b_{n,k}\rho_{n,k}^+=b^+_{n,k}$ and $b_{n,k}\rho^-_{n,k}=-b^-_{n,k}$ for all $n$ and $k$.

If $x\in\mathcal D(B)$ and $|x|\le u$, then $\|x\|\2u\le1$ and, for each $n$, $|(Bx)_n|\le (|B||x|)_n\le(|B|u)_n$ so we get $\|Bx\|\1v\le\||B|u\|\1v$. By positive homogeneity of the norm, $A(B)\le\||B|u\|\1v$. 

Fix $n$ and $K$. Define a sequence $x$ by setting $x_k=(\rho^+-\rho^-)_{n,k}u_k$ if $k\le K$ and $x_k=0$ if $k>K$. Then $x\in\mathcal D(B)$ and $\|x\|\2u\le1$. Since $|b_{n,k}|=b_{n,k}(\rho^+-\rho^-)_{n,k}$ for all $k$, we have
\[
v_n\sum_{k=1}^K|b_{n,k}|u_k
=v_n(Bx)_n\le\|Bx\|\1v\le A(B).
\]
Letting $K\to\infty$ and taking the supremum over all $n$ we get
\[
\||B|u\|\1v=\sup_n v_n(|B|u)_n\le A(B).
\]
This proves \ref{p1}.

If $x\in\ell^+\cap\mathcal D(B)$ and $x\le u$, then $\|x\|\2u\le1$ and, for each $n$, $0\le(B^+x)_n\le (B^+u)_n$ and $0\le(B^-x)_n\le (B^-u)_n$. If $(B^+u)_n$ and $(B^-u)_n$ are both finite, then 
\[
|(Bx)_n|=|(B^+x)_n-(B^-x)_n|\le \max((B^+u)_n,(B^-u)_n),
\]
an inequality that also holds if $(B^+u)_n=\infty$ or $(B^-u)_n=\infty$. 
Multiplying both sides by $v_n$, taking the supremum over all $n$, and interchanging the supremum and the maximum, yields
\[
\|Bx\|\1v\le\max(\|B^+u\|\1v, \|B^-u\|\1v).
\]
By positive homogeneity of the norm, 
\[
A^+(B)\le\max(\|B^+u\|\1v, \|B^-u\|\1v).
\] 

Fix $n$ and $K$. First, define a sequence $x$ by setting $x_k=\rho^+_{n,k}u_k$ if $k\le K$ and $x_k=0$ if $k>K$. Then $x\in\mathcal D(B)$ and $\|x\|\2u\le1$. Since $b^+_{n,k}=b_{n,k}\rho^+_{n,k}$,
\[
v_n\Big|\sum_{k=1}^Kb_{n,k}^+ u_k\Big|
=v_n|(Bx)_n|\le\|Bx\|\1v\le A^+(B).
\]
Next, define a sequence $x$ by setting $x_k=\rho^-_{n,k}u_k$ if $k\le K$ and $x_k=0$ if $k>K$. Then $x\in\mathcal D(B)$ and $\|x\|\2u\le1$. Since $b^-_{n,k}=-b_{n,k}\rho^-_{n,k}$,
\[
v_n\Big|\sum_{k=1}^Kb_{n,k}^- u_k\Big|
=v_n|(Bx)_n|\le\|Bx\|\1v\le A^+(B).
\]
Letting $K\to\infty$ and taking the supremum over all $n$ in the two estimates above, we get 
\[
\max(\|B^+u\|\1v, \|B^-u\|\1v)\le A^+(B).
\]
This proves \ref{p2}.

To prove \ref{p3}, suppose each row of $B$ has positives before negatives and a nonnegative sum. Fix $n$ and set $m=\sup\{k:b_{n,k}>0\}$, taking $\sup\emptyset=0$ if necessary. If $m=0$, then $b_{n,k}\le0$ for all $k$, but the $n$th row of $B$ has a nonnegative sum so $b_{n,k}=0$ for all $k$. If $m=\infty$, then $b_{n,k}\ge0$ for all $k$ because the $n$th row of $B$ has positives before negatives. In the remaining case, $m\in\mathbb Z^+$, $b_{n,k}\ge0$ for $k\le m$ and $b_{n,k}\le0$ for $k>m$. This implies that for all $x\in\ell^\downarrow\cap\mathcal D(B)$, $b_{n,k}x_k\ge b_{n,k}x_m$ for all $k$ and so $(Bx)_n\ge x_m\sum_{k=1}^\infty b_{n,k}\ge0$, because the $n$th row of $B$ has a nonnegative sum. In all three cases we get $(Bx)_n\ge0$.

If $x\in\ell^\downarrow\cap\mathcal D(B)$ and $x\le u$, then $x\le u^\downarrow$ by Lemma \ref{downandup}. Therefore,
$$
v_n|(Bx)_n|=v_n(Bx)_n\le v_n(B^+x)_n\le v_n(B^+(u^\downarrow))_n.
$$
Taking the supremum over all $n$, we get $\|Bx\|\1v\le\|B^+(u^\downarrow)\|\1v$. Positive homogeneity of the norm shows that 
$A^\downarrow(B)\le\|B^+(u^\downarrow)\|\1v$.

Fix $n$ and $K$, and define $m$ as above. Define $x$ by setting $x_k=u_k^\downarrow$ if $k\le \min(m,K)$ and $x_k=0$ otherwise. Then $x\in\ell^{\downarrow}\cap\mathcal D(B)$ and, by Lemma \ref{downandup}, $\|x\|\2u=\|x\|\2{u^\downarrow}\le1$. We have seen that $b^+_{n,k}=b_{n,k}$ for $k\le m$ and $b^+_{n,k}=0$ for $k>m$. Therefore,
\[
v_n\sum_{k=1}^Kb_{n,k}^+ u_k^\downarrow
=v_n(Bx)_n\le\|Bx\|\1v\le A^\downarrow(B).
\]
Letting $K\to\infty$ we get 
\[\| B^+(u^\downarrow)\|\1v\le A^\downarrow(B).
\]

To prove \ref{p4}, suppose each row of $B$ has negatives before positives and a nonnegative sum. Fix $n$ and set $m=\sup\{k:b_{n,k}<0\}$, taking $\sup\emptyset=0$ if necessary. If $m=\infty$, then $b_{n,k}\le0$ for all $k$ because the $n$th row of $B$ has negatives before positives, but the $n$th row of $B$ has a nonnegative sum so $b_{n,k}=0$ for all $k$. If $m=0$, then $b_{n,k}\ge0$ for all $k$. In the remaining case, $m\in\mathbb Z^+$, $b_{n,k}\le0$ for $k\le m$ and $b_{n,k}\ge0$ for $k>m$. This implies that for all $x\in\ell^\uparrow\cap\mathcal D(B)$, $b_{n,k}x_k\ge b_{n,k}x_m$ for all $k$ and so $(Bx)_n\ge x_m\sum_{k=1}^\infty b_{n,k}\ge0$, because the $n$th row of $B$ has a nonnegative sum. In all three cases we get $(Bx)_n\ge0$.

If $x\in\ell^\uparrow\cap\mathcal D(B)$ and $x\le u$, then $x\le u^\uparrow$ by Lemma \ref{downandup}. Therefore,
$$
v_n|(Bx)_n|=v_n(Bx)_n\le v_n(B^+x)_n\le v_n(B^+(u^\uparrow))_n.
$$
Taking the supremum over all $n$, we get $\|Bx\|\1v\le\|B^+(u^\uparrow)\|\1v$. Positive homogeneity of the norm shows that 
$A^\uparrow(B)\le\|B^+(u^\uparrow)\|\1v$.

Case 1. Every row of $B$ has a finite (nonnegative) sum. First we show that $\mathcal D(B)$ contains every nonnegative, bounded sequence. Suppose $x$ is such a sequence and choose $P$ so that $0\le x_k\le P$ for all $k$. Fix $n$ and let $m=\sup\{k:b_{n,k}<0\}$ again. As we have seen, if $m=\infty$, then $b_{n,k}=0$ for all $k$, so $\sum_{k=1}^\infty b_{n,k}x_k$ is trivially convergent. Otherwise, $\sum_{k=1}^Kb_{n,k}x_k$ is a nondecreasing for $K>m$ and is bounded above by
\[
\sum_{k=1}^mb_{n,k}x_k + P\sum_{k=m+1}^\infty b_{n,k}<\infty.
\] 
Again, $\sum_{k=1}^\infty b_{n,k}x_k$ is convergent. Since $n$ was arbitrary, $x\in\mathcal D(B)$.

Now fix $n$ and a real number $P$. Let $m=\sup\{k:b_{n,k}<0\}$. Define $x$ by setting $x_k=0$ if $k\le m$ and $x_k=\min(u_k^\uparrow,P)$ if $k> m$. Since $x$ is bounded above by $P$, $x\in\ell^\uparrow\cap\mathcal D(B)$. Moreover, $|x|\le u^\uparrow$ so $\|x\|\2u=\|x\|\2{u\uparrow}\le 1$ by Lemma \ref{downandup}. Therefore,
$$
v_n\sum_{k=1}^\infty b_{n,k}^+\min(u_k^\uparrow, P)=v_n\sum_{k=1}^\infty b_{n,k}x_k\le\|Bx\|\1v\le A^\uparrow(B).
$$
Taking the limit as $P\to\infty$, we get 
\[
v_n\sum_{k=1}^\infty b_{n,k}^+u_k^\uparrow \le A^\uparrow(B).
\]
Taking the supremum over all $n$, we get $\|B^+(u^\uparrow)\|\1v\le A^\uparrow(B)$.

Case 2. For some $n$, the $n$th row of $B$ has an infinite sum. Since its sum is nonnegative by hypothesis, the sum is $\infty$. Since this row has negatives before positives, there exists an $m$ such that $b_{n,k}\ge0$ when $k\ge m$. If $x\in \ell^\uparrow$ is not the zero sequence, then there exists a $K\ge m$ such that $x_K>0$. Therefore,
\[
\sum_{k=K}^\infty b_{n,k}x_k \ge x_K\sum_{k=K}^\infty b_{n,k}=\infty.
\]
Thus, $x\notin \mathcal D(B)$. We conclude that $\ell^\uparrow\cap\mathcal D(B)$ contains only the zero sequence, so $A^\uparrow(B)=0$.
\end{proof}

\subsection{The Ces\`aro and Copson operators}\label{specificsCC*}

The Ces\`aro matrix $C$ is nonnegative so all four parts of Theorem \ref{NormOfB} apply.
\begin{corollary}\label{SpecC}  Let $u,v\in\ell^+$. The inequality 
\[
\sup_n\Big|\frac1n\sum_{k=1}^nx_k\Big|v_n\le A\sup_k\frac{|x_k|}{u_k}
\] 
holds for all real sequences $x$ with $A=A(C)$; for all nonnegative sequences $x$ with $A=A^+(C)$; for all nonnegative, nonincreasing sequences $x$ with $A=A^\downarrow(C)$; and for all nonnegative, nondecreasing sequences $x$ with $A=A^\uparrow(C)$. In each case the constant $A$ is best possible. Here
\begin{align*}
A(C)=A^+(C)&=\|Cu\|\1v=
\sup_n\frac{v_n}n\sum_{k=1}^nu_k;\\
A^\downarrow(C)&=\|C(u^\downarrow)\|\1v=
\sup_n\frac{v_n}n\sum_{k=1}^n\min_{j\le k}u_j;\\
A^\uparrow(C)&=\|C(u^\uparrow)\|\1v=
\sup_n\frac{v_n}n\sum_{k=1}^n\inf_{j\ge k}u_j.
\end{align*}
\end{corollary}

Even using the formulas from this corollary, the operator norms of $C$ as a map on cones in a power weighted $\ell^\infty$ space requires some work to simplify. This is done in the next theorem.
\begin{theorem}\label{SpecCpower} Let $\alpha\in\mathbb R$. The inequality 
\[
\sup_n\Big|\frac1n\sum_{k=1}^nx_k\Big|n^\alpha\le A\sup_k|x_k|k^\alpha
\] 
holds for all real sequences $x$ if and only if it holds for all nonnegative sequences $x$ if and only if it holds for all nonnegative, nonincreasing sequences $x$. In this case the best constant $A$ is
\[
A
=\begin{cases}
1,&\alpha<0;\\
\dfrac1{1-\alpha},&0\le\alpha<1;\\
\infty,&\alpha\ge1.\end{cases}
\]
The inequality holds for all nonnegative, nondecreasing sequences $x$ with best constant 
\[
A
=\begin{cases}1,&\alpha\le0;\\0,&\alpha>0.\end{cases}
\]
\end{theorem}
\begin{proof} Take $u_k=k^{-\alpha}$ and $v_n=n^\alpha$ in Corollary \ref{SpecC} to get 
\begin{align*}
A(C)=A^+(C)&=\sup_nn^{\alpha-1}\sum_{k=1}^nk^{-\alpha};\\
A^\downarrow(C)&=
\sup_nn^{\alpha-1}\sum_{k=1}^n\min_{j\le k}j^{-\alpha};\\
A^\uparrow(C)&=
\sup_nn^{\alpha-1}\sum_{k=1}^n\inf_{j\ge k}j^{-\alpha}.
\end{align*}
We will make use of Proposition 3 of \cite{BJ}, which shows that 
 \[
n^{\alpha-1}\sum_{k=1}^nk^{-\alpha}
\]
increases with $n$ when $\alpha\ge0$ and decreases with $n$ when $\alpha\le0$. If $\alpha\ge0$, then $\min_{j\le k}j^{-\alpha}=k^{-\alpha}$ so $A(C)=A^+(C)=A^\downarrow(C)$ and their common value is
\[
\lim_{n\to\infty}\frac1n\sum_{k=1}^n\Big(\frac kn\Big)^{-\alpha}=\int_0^1x^{-\alpha}\,dx=\begin{cases}
\dfrac1{1-\alpha},&0\le\alpha<1;\\
\infty,&\alpha\ge1.\end{cases}
\]
If $\alpha<0$, then $\min_{j\le k}j^{-\alpha}=1$ so
\[
A(C)=A^+(C)=\sup_nn^{\alpha-1}\sum_{k=1}^nk^{-\alpha}=1\ \ \text{and}\ \ 
A^\downarrow(C)=\sup_nn^{\alpha}=1.
\]

If $\alpha>0$, then $\inf_{j\ge k}j^{-\alpha}=0$ and if $\alpha\le0$, then $\inf_{j\ge k}j^{-\alpha}=k^{-\alpha}$. Therefore, $A^\uparrow(C)=0$ when $\alpha>0$ and 
\[
A^\uparrow(C)=\sup_nn^{\alpha-1}\sum_{k=1}^nk^{-\alpha}=1
\]
when $\alpha\le0$.
\end{proof}

The Copson matrix $C^*$ is nonnegative so all four parts of Theorem \ref{NormOfB} apply, although the fourth part applies trivially. Recall that $\mathcal D(C^*)$ consists of all real sequences $x$ for which $\sum_{k=1}^\infty\frac{x_k}k$ converges in $\mathbb R$.
\begin{corollary}\label{SpecC*}  Let $u,v\in\ell^+$. The inequality 
\[
\sup_n\Big|\sum_{k=n}^\infty\frac{x_k}k\Big|v_n\le A\sup_k\frac{|x_k|}{u_k}
\] 
holds for all real sequences $x\in\mathcal D(C^*)$ with $A=A(C^*)$; for all nonnegative sequences $x\in\mathcal D(C^*)$ with $A=A^+(C^*)$; for all nonnegative, nonincreasing sequences $x\in\mathcal D(C^*)$ with $A=A^\downarrow(C^*)$; and for all nonnegative, nondecreasing sequences $x\in\mathcal D(C^*)$ with $A=A^\uparrow(C^*)$. In each case the constant $A$ is best possible. Here
\begin{align*}
A(C^*)=A^+(C^*)&=\|C^*u\|\1v=
\sup_nv_n\sum_{k=n}^\infty\frac{u_k}k;\\
A^\downarrow(C^*)&=\|C^*(u^\downarrow)\|\1v=
\sup_nv_n\sum_{k=n}^\infty\frac1k\min_{j\le k}u_j;\\
A^\uparrow(C^*)&=0. \quad(\ell^\uparrow\cap\mathcal D(C^*)=\{0\}.)
\end{align*}
\end{corollary}

The operator norms of $C^*$ as a map on cones in a power weighted $\ell^\infty$ space are given in the next theorem. We use $\zeta$ to denote the Riemann zeta function.
\begin{theorem}\label{SpecC*power} Let $\alpha\in\mathbb R$. The inequality 
\[
\sup_n\Big|\sum_{k=n}^\infty\frac{x_k}k\Big|n^\alpha\le A\sup_k|x_k|k^\alpha
\]  
holds for all real sequences $x\in\mathcal D(C^*)$ if and only if it holds for all nonnegative sequences $x\in\mathcal D(C^*)$ if and only if it holds for all nonnegative, nonincreasing sequences $x\in\mathcal D(C^*)$. In this case the best constant $A$ is
\[
A
=\begin{cases}
\infty,&\alpha\le0;\\
\zeta(\alpha+1),&\alpha>0.\end{cases}
\]
Except for the zero sequence, there are no nonnegative, nondecreasing sequences in $\mathcal D(C^*)$. The inequality holds for the zero sequence $x$ with best constant $A=0$.
\end{theorem}
\begin{proof} Take $u_k=k^{-\alpha}$ and $v_n=n^\alpha$ in Corollary \ref{SpecC*} to get 
\begin{align*}
A(C^*)=A^+(C^*)&=\sup_nn^\alpha\sum_{k=n}^\infty k^{-\alpha-1};\\
A^\downarrow(C^*)&=
\sup_nn^\alpha\sum_{k=n}^\infty\frac1k\min_{j\le k}j^{-\alpha}.
\end{align*}
If $\alpha>0$, then $\min_{j\le k}j^{-\alpha}=k^{-\alpha}$ so $A(C^*)=A^+(C^*)=A^\downarrow(C^*)$; their common value is
\[
\sup_nn^\alpha\sum_{k=n}^\infty k^{-\alpha-1}\ge\sum_{k=1}^\infty k^{-\alpha-1}=\zeta(\alpha+1).
\]
We show this is actually equality by supplying a proof of the first inequality from Remark 4.10 of \cite{B}, namely, that $n^\alpha\sum_{k=n}^\infty k^{-\alpha-1}$ decreases with $n$: The derivative of $\log(x^{\alpha+1}(x^{-\alpha}-(x+1)^{-\alpha}))$ is 
\[
\frac{\big(1+\frac1x\big)^{\alpha+1}-\big(1+\frac{\alpha+1}x\big)}{(1+x)^{\alpha+1}(x^{-\alpha}-(x+1)^{-\alpha})},
\]
which is positive for $x>0$ by Bernoulli's inequality. Thus,
\[
a_k=\frac1{k^{\alpha+1}(k^{-\alpha}-(k+1)^{-\alpha})}
\]
is a decreasing sequence, and so is its moving average
\[
\frac{\sum_{k=n}^\infty a_k(k^{-\alpha}-(k+1)^{-\alpha})}{\sum_{k=n}^\infty (k^{-\alpha}-(k+1)^{-\alpha})}=n^\alpha\sum_{k=n}^\infty k^{-\alpha-1}.
\]

If $\alpha\le0$, then $\min_{j\le k}j^{-\alpha}=1$ and we have
\[
A(C^*)=A^+(C^*)\ge A^\downarrow(C^*)=\sup_nn^\alpha\sum_{k=n}^\infty \frac1k=\infty.
\]

The final statement of the theorem is evident.
\end{proof}

\subsection{The Ces\`aro and Copson operators minus identity}\label{specificsCC*-I} The matrices we consider here are,
\[
C-I={\scriptstyle\left(\begin{smallmatrix}
0&&&&\\
\vphantom{\frac12}\frac12&-\frac12&&&\\
\vphantom{\frac12}\frac13&\frac13&-\frac23&&\\
\vphantom{\frac12}\frac14&\frac14&\frac14&-\frac34&\\
\vphantom{\frac12}\vdots&\vdots&\vdots&\vdots&\ddots\\
&\hphantom{-1}&\hphantom{-1}&\hphantom{-1}&\hphantom{-1}
\end{smallmatrix}\right)}
\quad\mbox{and}\quad 
C^*-I={\scriptstyle\left(\begin{smallmatrix}  
\vphantom{\frac12}0&\frac12&\frac13&\frac14&\dots\\
\vphantom{\frac12}&-\frac12&\frac13&\frac14&\dots\\
\vphantom{\frac12}&&-\frac23&\frac14&\dots\\
\vphantom{\frac12}&&&-\frac34&\dots\\
&&&&\ddots\\
&\hphantom{-1}&\hphantom{-1}&\hphantom{-1}&\hphantom{-1}
\end{smallmatrix}\right)}.
\]
Parts \ref{p1}, \ref{p2}, and \ref{p3} of Theorem \ref{NormOfB} apply to $C-I$ and part \ref{p4} applies to $I-C$. (See Remark \ref{preproc}.)
\begin{corollary}\label{SpecCI} Let $u,v\in\ell^+$. The inequality 
\[
\sup_n\Big|\Big(\frac1n\sum_{k=1}^nx_k\Big)-x_n\Big|v_n\le A\sup_k\frac{|x_k|}{u_k}
\] 
holds for all real sequences $x$ with $A=A(C-I)$; for all nonnegative sequences $x$ with $A=A^+(C-I)$; for all nonnegative, nonincreasing sequences $x$ with $A=A^\downarrow(C-I)$ and for all nonnegative, nondecreasing sequences $x$ with $A=A^\uparrow(I-C)$. In each case the constant $A$ is best possible. Here
\begin{align*}
A(C-I)&=\||C-I|u\|\1v=\sup_n\frac{v_n}n\Big((n-1)u_n+\sum_{k=1}^{n-1}u_k\Big);\\
A^+(C-I)&=\max(\|(C-I)^+u\|\1v,\|(C-I)^-u\|\1v)\\&=\sup_n\frac{v_n}n\max\Big((n-1)u_n,\sum_{k=1}^{n-1}u_k\Big);\\
A^\downarrow(C-I)&=\|(C-I)^+(u^\downarrow)\|\1v=\sup_n\frac{v_n}n\sum_{k=1}^{n-1}\min_{j\le k}u_j;\\
A^\uparrow(I-C)&=\|(I-C)^+(u^\uparrow)\|\1v=\sup_n\frac{v_n}n(n-1)\inf_{j\ge n}u_j.
\end{align*}
\end{corollary}
 
Proposition 3.5 of \cite{BS20} may be compared with the case $0\le \alpha<1$ of the next theorem: The norm of $C-I$ restricted to the cone of nonnegative, nonincreasing sequences coincides with the norm of $H-I$ restricted to the cone of  nonnegative, nonincreasing functions. 

On power-weighted $\ell^\infty$, $C-I$ exhibits different behavior on all four different cones, the cone of real sequences, the cone of nonnegative sequences, the cone of nonnegative, nonincreasing sequences and the cone of nonnegative, nondecreasing sequences. The dependence of the operator norm on the power gets particularly interesting for the third cone. 

\begin{theorem}\label{SpecCIpower} Let $\alpha\in\mathbb R$, set $s_1=-\infty$ and set $s_m=1+\frac{\log(1-1/m)}{\log(1+1/m)}$ for $m=2,3,\dots$. The inequality 
\[
\sup_n\Big|\Big(\frac1n\sum_{k=1}^nx_k\Big)-x_n\Big|n^\alpha\le A\sup_k|x_k|k^\alpha
\]  
holds for all real sequences $x$ with best constant
\[
A=\begin{cases}\dfrac{2-\alpha}{1-\alpha},&\alpha<1;\\\infty,&\alpha\ge1.\end{cases}
\]
It holds for all nonnegative sequences $x$ with best constant
\[
A=\begin{cases}1,&\alpha<0;\\
\dfrac1{1-\alpha},&0\le\alpha<1;\\
\infty,&\alpha\ge1.\end{cases}
\]
It holds for all nonnegative, nonincreasing sequences $x$ with best constant
\[
A=\begin{cases}(m+1)^{\alpha-1}m,&s_m<\alpha\le s_{m+1},\ m=1,2,3,\dots;\\
\dfrac1{1-\alpha},&0\le\alpha<1;\\
\infty,&\alpha\ge1.\end{cases}
\]
It holds for all nonnegative, nondecreasing sequences $x$ with best constant
\[
A=\begin{cases}1,&\alpha\le0;\\
0,&\alpha>0.\end{cases}
\]
\end{theorem}
\begin{proof} Take $u_k=k^{-\alpha}$ and $v_n=n^\alpha$ in Corollary \ref{SpecCI} to get 
\begin{align*}
A(C-I)&=\sup_n\Big(1-\frac1n+n^{\alpha-1}\sum_{k=1}^{n-1} k^{-\alpha}\Big);\\
A^+(C-I)&=\sup_n\max\Big(1-\frac1n,n^{\alpha-1}\sum_{k=1}^{n-1} k^{-\alpha}\Big);\\
A^\downarrow(C-I)&=
\sup_nn^{\alpha-1}\sum_{k=1}^{n-1}\min_{j\le k}j^{-\alpha};\\
A^\uparrow(I-C)&=
\sup_nn^{\alpha-1}(n-1)\inf_{j\ge n}j^{-\alpha}.
\end{align*}
Proposition 4 in \cite{BJ} shows that $n^{\alpha-1}\sum_{k=1}^{n-1}k^{-\alpha}$ increases with $n$ for all $\alpha\in\mathbb R$. It tends to 
\[
\int_0^1 x^{-\alpha}\,dx=\begin{cases}\dfrac1{1-\alpha},&\alpha<1;\\\infty,&\alpha\ge1.\end{cases}
\]
The first two statements of the theorem follow.
 
If $\alpha \ge0$, then $\min_{j\le k}j^{-\alpha}=k^{-\alpha}$ so $A^\downarrow(C-I)=A^+(C-I)$. If $\alpha<0$, then $\min_{j\le k}j^{-\alpha}=1$ so $A^\downarrow(C-I)=\sup_n n^{\alpha-1}(n-1)$. Consider the function $g(x)=x^{\alpha-1}(x-1)$ for $x\ge1$. Looking at $g'(x)$ we find that $g$ is strictly increasing on $(0,1-1/\alpha)$ and strictly decreasing on $(1-1/\alpha,\infty)$. It follows that a positive integer $m$ satisfies $\sup_n g(n)=g(m+1)$ if and only if $g(m+1)\ge g(m)$ and $g(m+1)\ge g(m+2)$. These two conditions may be expressed as $s_m\le\alpha\le s_{m+1}$. 

If $\alpha>0$, then $\inf_{j\ge n}j^{-\alpha}=0$ so $A^\uparrow(I-C)=0$. If $\alpha\le0$, then $\inf_{j\ge n}j^{-\alpha}=n^{-\alpha}$ so $A^\uparrow(I-C)=\sup_n (n-1)/n=1$.
\end{proof}

Parts \ref{p1}, \ref{p2}, and \ref{p4} of Theorem \ref{NormOfB} apply to $C^*-I$, although part \ref{p4} gives a trivial result. Note that $\mathcal D(C^*-I)=\mathcal D(C^*)$.

\begin{corollary}\label{SpecC*I} Let $u,v\in\ell^+$. The inequality 
\[
\sup_n\Big|\Big(\sum_{k=n}^\infty\frac{x_k}k\Big)-x_n\Big|v_n\le A\sup_k\frac{|x_k|}{u_k}
\] 
holds for all real sequences $x\in\mathcal D(C^*)$ with $A=A(C^*-I)$; for all nonnegative sequences $x\in\mathcal D(C^*)$ with $A=A^+(C^*-I)$; and for all nonnegative, nondecreasing sequences $x\in\mathcal D(C^*)$ with $A=A^\uparrow(C^*-I)$. In each case the constant $A$ is best possible. Here
\begin{align*}
A(C^*-I)&=\||C^*-I|u\|\1v=\sup_nv_n\Big(\frac{n-1}nu_n+\sum_{k=n+1}^\infty\frac{u_k}k\Big);\\
A^+(C^*-I)&=\max(\|(C^*-I)^+u\|\1v,\|(C^*-I)^-u\|\1v)\\&=\sup_nv_n\max\Big(\frac{n-1}nu_n,\sum_{k=n+1}^{\infty}\frac{u_k}k\Big);\\
A^\uparrow(C^*-I)&=0. \text{ (The first row of $C^*-I$ has an infinite sum.)}
\end{align*}
\end{corollary}

We omit the trivial case when considering the power-weighted inequalities.

\begin{theorem}\label{SpecC*Ipower} Let $\alpha\in\mathbb R$. The inequality 
\[
\sup_n\Big|\Big(\sum_{k=n}^\infty\frac{x_k}k\Big)-x_n\Big|n^\alpha\le A\sup_k|x_k|k^\alpha
\]  
holds for all real sequences $x\in\mathcal D(C^*)$ with
\[
A=\begin{cases}\infty,&\alpha\le0;\\
1+\dfrac1\alpha,&\alpha>0.\end{cases}
\]
It holds for all nonnegative sequences $x\in\mathcal D(C^*)$ with
\[
A=\begin{cases}\infty,&\alpha\le0;\\
\dfrac1\alpha,&0<\alpha<1;\\
1,&\alpha\ge1.\end{cases}
\]
In each case the value of $A$ is best possible.
\end{theorem}
\begin{proof} Take $u_k=k^{-\alpha}$ and $v_n=n^\alpha$ in Corollary \ref{SpecC*I} to get 
\begin{align*}
A(C^*-I)&=\sup_n\Big(1-\frac1n+n^{\alpha}\sum_{k=n+1}^\infty k^{-\alpha-1}\Big);\\
A^+(C^*-I)&=\sup_n\max\Big(1-\frac1n,n^{\alpha}\sum_{k=n+1}^\infty k^{-\alpha-1}\Big).
\end{align*}
If $\alpha\le0$ the sum diverges so both of these are infinite. If $\alpha>0$ we need the second inequality from Remark 4.10 of \cite{B}: $n^\alpha\sum_{k=n+1}^\infty k^{-\alpha-1}$ increases with $n$. The derivative of $\log(x^{\alpha+1}((x-1)^{-\alpha}-x^{-\alpha}))$ is 
\[
\frac{\big(1-\frac{\alpha+1}x\big)-\big(1-\frac1x\big)^{\alpha+1}}{(x-1)^{\alpha+1}((x-1)^{-\alpha}-x^{-\alpha})},
\]
which is negative for $x>1$ by Bernoulli's inequality. Thus,
\[
a_k=\frac1{k^{\alpha+1}((k-1)^{-\alpha}-k^{-\alpha})}
\]
is an increasing sequence, and so is its moving average
\[
\frac{\sum_{k=n+1}^\infty a_k((k-1)^{-\alpha}-k^{-\alpha})}{\sum_{k=n+1}^\infty ((k-1)^{-\alpha}-k^{-\alpha})}=n^\alpha\sum_{k=n+1}^\infty k^{-\alpha-1}.
\]
We recognize these as (improper) Riemann sums, and get
\[
n^\alpha\sum_{k=n+1}^\infty k^{-\alpha-1}
=\frac1n\sum_{k=n+1}^\infty \Big(\frac kn\Big)^{-\alpha-1}\to\int_1^\infty x^{-\alpha-1}\,dx=\frac1\alpha
\]
as $n\to\infty$. Therefore $A(C^*-I)=1+1/\alpha$ and $A^+(C^*-I)=\max(1,1/\alpha)$. This completes the proof.
\end{proof}

\subsection{Two required operators}\label{specificsreductions} 

The next two operators appear in Theorems \ref{CC*reduction} and \ref{C*Creduction}. Their operators norms are needed to complete the work on \eqref{CandC*}. In matrix form, they are
\[
C-S^*={\scriptstyle\left(\begin{smallmatrix}  
\vphantom{\frac12}1&-1&&&&\\
\vphantom{\frac12}\frac12&\frac12&-1&&&\\
\vphantom{\frac12}\frac13&\frac13&\frac13&-1&&\\
\vphantom{\frac12}\frac14&\frac14&\frac14&\frac14&-1&\\
\vdots&\vdots&\vdots&\vdots&\vdots&\ddots\\
&\hphantom{-1}&\hphantom{-1}&\hphantom{-1}&\hphantom{-1}
\end{smallmatrix}\right)}
\quad\mbox{and}\quad 
(C^*-S)D={\scriptstyle\left(\begin{smallmatrix}  
\vphantom{\frac12}\frac12&\frac16&\frac1{12}&\frac1{20}&\dots\\
\vphantom{\frac12}-\frac12&\frac16&\frac1{12}&\frac1{20}&\dots\\
\vphantom{\frac12}&-\frac13&\frac1{12}&\frac1{20}&\dots\\
\vphantom{\frac12}&&-\frac14&\frac1{20}&\dots\\
\vphantom{\frac12}&&&\ddots&\ddots\\
&\hphantom{-1}&\hphantom{-1}&\hphantom{-1}&\hphantom{-1}
\end{smallmatrix}\right)}.
\]
Parts \ref{p1}, \ref{p2}, and \ref{p3} of Theorem \ref{NormOfB} apply to $C-S^*$ and part \ref{p4} applies to $S^*-C$.
\begin{corollary}\label{SpecCS*} Let $u,v\in\ell^+$. The inequality 
\[
\sup_n\Big|\Big(\frac1n\sum_{k=1}^nx_k\Big)-x_{n+1}\Big|v_n\le A\sup_k\frac{|x_n|}{u_n}
\] 
holds for all real sequences $x$ with $A=A(C-S^*)$; for all nonnegative sequences $x$ with $A=A^+(C-S^*)$; for all nonnegative, nonincreasing sequences $x$ with $A=A^\downarrow(C-S^*)$ and for all nonnegative, nondecreasing sequences $x$ with $A=A^\uparrow(S^*-C)$. In each case the constant $A$ is best possible. Here
\begin{align*}
A(C-S^*)&=\||C-S^*|u\|\1v=\sup_nv_n\Big(u_{n+1}+\frac1n\sum_{k=1}^nu_k\Big);\\
A^+(C-S^*)&=\max(\|(C-S^*)^+u\|\1v,\|(C-S^*)^-u\|\1v)\\&=\sup_nv_n\max\Big(\frac1n\sum_{k=1}^nu_k,u_{n+1}\Big);\\
A^\downarrow(C-S^*)&=\|(C-S^*)^+(u^\downarrow)\|\1v=\sup_n\frac{v_n}n\sum_{k=1}^n\min_{j\le k}u_j;\\
A^\uparrow(S^*-C)&=\|(S^*-C)^+(u^\uparrow)\|\1v=\sup_nv_n\inf_{j\ge n+1}u_j.
\end{align*}
\end{corollary}

Parts \ref{p1}, \ref{p2}, and \ref{p4} of Theorem \ref{NormOfB} apply to $(C^*-S)D$ and part \ref{p3} applies to $(S-C^*)D$. Note that $\mathcal D((C^*-S)D)=\mathcal D((S-C^*)D)=\mathcal D(C^*D)$.
\begin{corollary}\label{SpecC*S} Let $u,v\in\ell^+$ and for convenience let $u_0=x_0=0$. The inequality 
\[
\sup_n\Big|\Big(\sum_{k=n}^\infty \frac{x_k}{k(k+1)}\Big)-\frac{x_{n-1}}n\Big|v_n\le A\sup_k\frac{|x_n|}{u_n}
\] 
holds for all real sequences $x\in\mathcal D(C^*D)$ with $A=A((C^*-S)D)$; for all nonnegative sequences $x\in\mathcal D(C^*D)$ with $A=A^+((C^*-S)D)$; for all nonnegative, nonincreasing sequences $x\in\mathcal D(C^*D)$ with $A=A^\downarrow((S-C^*)D)$; and for all nonnegative, nondecreasing sequences $x\in\mathcal D(C^*D)$ with $A=A^\uparrow((C^*-S)D)$. In each case the constant $A$ is best possible. Here
\begin{align*}
A((C^*-S)D)&=\||(C^*-S)D|u\|\1v\\
&=\sup_nv_n\Big(\frac{u_{n-1}}n+\sum_{k=n}^\infty\frac{u_k}{k(k+1)}\Big);\\
A^+((C^*-S)D)&=\max(\|((C^*-S)D)^+u\|\1v,\|((C^*-S)D)^-u\|\1v)\\
&=\sup_nv_n\max\Big(\frac{u_{n-1}}n,\sum_{k=n}^\infty\frac{u_k}{k(k+1)}\Big);\\
A^\downarrow((S-C^*)D)&=\|((S-C^*)D)^+(u^\downarrow)\|\1v=\sup_n\frac{v_n}n\min_{j\le n-1}u_j;\\
A^\uparrow((C^*-S)D)&=\|((C^*-S)D)^+(u^\uparrow)\|\1v\\
&=\sup_nv_n\sum_{k=n}^\infty\frac1{k(k+1)}\inf_{j\ge k}u_j.
\end{align*}
\end{corollary}

We forgo an investigation of the power-weighted case for these operators. Their principal interest is their use in the proof of Theorem \ref{BestConstants} and the special case is not required there. Theorems \ref{powerCC*} and \ref{powerC*C} deduce the power-weighted case of Theorem \ref{BestConstants} directly.

\section{Best constants in the two-operator inequalities}

Combining Theorems \ref{CC*reduction} and \ref{C*Creduction} with the formulas given in the previous subsection for  
\[
A(C-S^*),\ A^\downarrow(C-S^*),\ A((C^*-S)D),\quad\mbox{and}\quad A^\uparrow((C^*-S)D)
\]
gives us answers to our original questions, the best constants in the inequalities of \eqref{CandC*}.

Fix $u,v\in\ell^+$. Let $A(C,C^*)$ and $A^+(C,C^*)$ denote the smallest $A\ge0$ such that inequality
\[
\|Cx\|\1v\le A\|C^*x\|\2u
\]
holds for all $x\in\mathcal D(C^*)$ and for all $x\in\ell^+\cap\mathcal D(C^*)$, respectively.
Similarly, let $A(C^*,C)$ and $A^+(C^*,C)$ denote the smallest $A\ge0$ such that inequality
\[
\|C^*x\|\1v\le A\|Cx\|\2u
\]
holds for all $x\in\mathcal D(C^*)$ and for all $x\in\ell^+\cap\mathcal D(C^*)$, respectively.

\begin{theorem}\label{BestConstants}  Let $u,v\in\ell^+$. Then, taking $u_0=0$, we have
\begin{align*} 
A(C,C^*)&=\sup_nv_n\Big(u_{n+1}+\frac1n\sum_{k=1}^nu_k\Big);\\
A^+(C,C^*)&=\sup_n\frac{v_n}n\sum_{k=1}^n\min_{j\le k}u_j;\\
A(C^*,C)&=\sup_nv_n\Big(\frac{n-1}nu_{n-1}+\sum_{k=n}^\infty\frac{u_k}{k+1}\Big);\\
A^+(C^*,C)&=\sup_nv_n\sum_{k=n}^\infty\frac1{k(k+1)}\inf_{j\ge k}ju_j.
\end{align*}
\end{theorem}
\begin{proof} Since \eqref{C-S*0} holds with $A=A(C-S^*)$, Theorem \ref{CC*reduction} shows that \eqref{CC*} does as well. Thus $A(C,C^*)\le A(C-S^*)$. On the other hand, by definition, \eqref{CC*} holds with $A=A(C,C^*)$ and by Theorem \ref{CC*reduction}, so does \eqref{C-S*0}. Fix $n$ and let 
\[
y=(u_1, u_2,\dots, u_n, -u_{n+1},0,0,\dots).
\]
Since $y_k\to0$ as $k\to\infty$ and $\|y\|\2u\le 1$, we get
\[
 v_n\Big(u_{n+1}+\frac1n\sum_{k=1}^nu_k\Big)
=v_n((C-S^*)y)_n\le \|(C-S^*)y\|\1v\le A(C,C^*).
\]
Using this in the formula for $A(C-S^*)$ from Corollary \ref{SpecCS*}  yields
\[
A(C-S^*)=\sup_nv_n\Big(u_{n+1}+\frac1n\sum_{k=1}^nu_k\Big)\le A(C,C^*).
\]
Therefore,
\[
A(C,C^*)=\sup_nv_n\Big(u_{n+1}+\frac1n\sum_{k=1}^nu_k\Big).
\]

Clearly, \eqref{C-S*d0} holds with $A=A^\downarrow(C-S^*)$ and, by Theorem \ref{CC*reduction}, so does \eqref{CC*+}. Thus $A^+(C,C^*)\le A^\downarrow(C-S^*)$. By definition, \eqref{CC*+} holds with $A=A^+(C,C^*)$ and Theorem \ref{CC*reduction}  shows that \eqref{C-S*d0} does also. Fix $n$ and let 
\[
y=(u_1^\downarrow, u_2^\downarrow,\dots, u_n^\downarrow,0,0,\dots).
\]
Then $y\in\ell^\downarrow$, $y_k\to0$ as $k\to\infty$ and, by Lemma \ref{downandup}, $\|y\|\2u=\|y\|\2{u^\downarrow}\le 1$. Therefore 
\[
 v_n\Big(\frac1n\sum_{k=1}^nu_k^\downarrow\Big)
=v_n((C-S^*)y)_n\le \|(C-S^*)y\|\1v\le A^+(C,C^*).
\]
The formula for $A^\downarrow(C-S^*)$ from Corollary \ref{SpecCS*} implies
\[
A^\downarrow(C-S^*)=\sup_n\frac{v_n}n\sum_{k=1}^n\min_{j\le k}u_j\le A^+(C,C^*).
\]
Therefore,
\[
A^+(C,C^*)=\sup_n\frac{v_n}n\sum_{k=1}^n\min_{j\le k}u_j.
\]

Let $w_k=ku_k$ for all $k$ and note that $\mathcal D((C^*-S)D)=\mathcal D(C^*D)$.

Replacing $u$ by $w$  in the formula for $A((C^*-S)D)$ from Corollary \ref{SpecC*S} shows that \eqref{C*-S0} holds with $A$ replaced by
\[
\tilde A=\sup_nv_n\Big(\frac{w_{n-1}}n+\sum_{k=n}^\infty\frac{w_k}{k(k+1)}\Big)=
\sup_nv_n\Big(\frac{n-1}nu_{n-1}+\sum_{k=n}^\infty\frac{u_k}{k+1}\Big).
\]
By Theorem \ref{C*Creduction}, \eqref{C*C} also holds with $A=\tilde A$. Thus $A(C^*,C)\le \tilde A$. Theorem \ref{C*Creduction} also shows that \eqref{C*C} and hence \eqref{C*-S0} holds with $A=A(C^*,C)$. Fix $n$ and $K>n$, and let 
\[
z=(0,0,\dots,0,-w_{n-1},w_n,w_{n+1},\dots w_K,0,0,\dots).
\]
Evidently, $z\in\mathcal D(C^*D)$, $(Dz)_k\to0$ as $k\to\infty$, and $\|z\|\2w\le1$, so we get $\|(C^*-S)Dz\|\1v\le A(C^*,C)$. It follows that
\[
 v_n\Big(\frac{w_{n-1}}n+\sum_{k=n}^K\frac{w_k}{k(k+1)}\Big)
=v_n((C^*-S)Dz)_n\\
\le A(C^*,C).
\]
Letting $K\to\infty$ and taking the supremum over $n$, gives $\tilde A\le A(C^*,C)$, and we conclude that $A(C^*,C)=\tilde A$.

Replacing $u$ by $w$ in the formula for $A^\uparrow((C^*-S)D)$ from Corollary \ref{SpecC*S} shows that \eqref{C*-Su0} holds with $A$ replaced by
\[
\tilde A^\uparrow=\sup_nv_n\sum_{k=n}^\infty\frac{w_k^\uparrow}{k(k+1)}=
\sup_nv_n\sum_{k=n}^\infty\frac1{k(k+1)}\inf_{j\ge k}ju_j.
\]
By Theorem \ref{C*Creduction}, \eqref{C*C+} also holds with $A=\tilde A^\uparrow$. Thus $A^+(C^*,C)\le \tilde A^\uparrow$. Theorem \ref{C*Creduction} also shows that \eqref{C*C+} and hence \eqref{C*-Su0} holds with $A=A^+(C^*,C)$. Fix $n$ and $K>n$, and let 
\[
z=(0,0,\dots,0,w_n^\uparrow,w_{n+1}^\uparrow,\dots w_K^\uparrow,w_K^\uparrow,w_K^\uparrow,\dots).
\]
Then $(C^*Dz)_K=\sum_{k=K}^\infty\frac{w_K^\uparrow}{k(k+1)}<\infty$ so $z\in \ell^\uparrow\cap\mathcal D(C^*D)$, $(Dz)_k\to0$ as $k\to\infty$, and, by Lemma \ref{downandup}, $\|z\|\2w=\|z\|\2{w^\uparrow}\le1$. Therefore,
\[
 v_n\sum_{k=n}^K\frac{w_k^\uparrow}{k(k+1)}
\le v_n((C^*-S)Dz)_n\\
\le \|(C^*-S)Dz\|\1v\\
\le A^+(C^*,C).
\]
Letting $K\to\infty$ and taking the supremum over $n$, we get $\tilde A^\uparrow\le A^+(C^*,C)$, and conclude that $A^+(C^*,C)=\tilde A^\uparrow$.
\end{proof}

We split the power-weighted inequalities for the two-operator inequalities into two theorems because the techniques of simplification differ. 
\begin{theorem}\label{powerCC*} Let $\alpha\in\mathbb R$. The inequality
\[
\sup_n\Big|\frac1n\sum_{k=1}^nx_k\Big|n^\alpha\le A\sup_n\Big|\sum_{k=n}^\infty\frac{x_k}k\Big|n^\alpha
\]
holds for all $x\in\mathcal D(C^*)$ with 
\[
A=\begin{cases}1+2^{-\alpha},&\alpha\le0;\\
\dfrac{2-\alpha}{1-\alpha},&0<\alpha<1;\\
\infty,&\alpha\ge1.\end{cases}
\]
It holds for all nonnegative $x\in\mathcal D(C^*)$ with
\[
A=\begin{cases}1,&\alpha\le0;\\
\dfrac1{1-\alpha},&0<\alpha<1;\\
\infty,&\alpha\ge1.\end{cases}
\]
In each case the constant $A$ is best possible.
\end{theorem} 
\begin{proof}Take $u_k=k^{-\alpha}$ and $v_n=n^\alpha$ in Theorem \ref{BestConstants} to see that the best constant $A$, taken over all $x\in\mathcal D(C^*)$, is
\[
A(C,C^*)=\sup_n\Big(\Big(\frac n{n+1}\Big)^\alpha+n^{\alpha-1}\sum_{k=1}^nk^{-\alpha}\Big)
\]
and the best constant $A$, taken over all nonnegative $x\in\mathcal D(C^*)$, is
\[
A^+(C,C^*)=\sup_nn^{\alpha-1}\sum_{k=1}^n\min_{j\le k}j^{-\alpha}.
\]
By Proposition 3 of \cite{BJ}, $n^{\alpha-1}\sum_{k=1}^nk^{-\alpha}$ decreases with $n$ when $\alpha\le 0$ and increases with $n$ when $\alpha>0$. The same is true of $\big(\frac n{n+1}\big)^\alpha$. So if $\alpha\le 0$, $A(C,C^*)=1+2^{-\alpha}$ and if $\alpha>0$, $\big(\frac n{n+1}\big)^\alpha\to1$ and
\[
n^{\alpha-1}\sum_{k=1}^nk^{-\alpha}=\frac1n\sum_{k=1}^n\Big(\frac kn\Big)^{-\alpha}\to\int_0^1 x^{-\alpha}\,dx=\begin{cases}\dfrac1{1-\alpha},&0<\alpha<1;\\
\infty,&\alpha\ge1,\end{cases}
\]
as $n\to\infty$ so 
\[
A(C,C^*)=\begin{cases}\dfrac{2-\alpha}{1-\alpha},&0<\alpha<1;\\
\infty,&\alpha\ge1,\end{cases}
\]

If $\alpha\le0$, then $\min_{j\le k}j^{-\alpha}=1$ so $A^+(C,C^*)= 1$. If $\alpha>0$, then $\min_{j\le k}j^{-\alpha}=k^{-\alpha}$ so as above we have
\[
A^+(C,C^*)=\sup_nn^{\alpha-1}\sum_{k=1}^nk^{-\alpha}=\begin{cases}\dfrac1{1-\alpha},&0<\alpha<1;\\
\infty,&\alpha\ge1.\end{cases}
\]
\end{proof}

In the case of nonnegative sequences, the best constants given above agree with those that appear in Theorem 3.7 of \cite{BS20} for the operators $H$ and $H^*$ on nonnegative functions, except when $\alpha<0$. The best constants given below, again for nonnegative sequences, agree with the corresponding results from Theorem 3.7 of \cite{BS20} for all values of $\alpha$.

\begin{theorem}\label{powerC*C} Let $\alpha\in\mathbb R$ and set $M_\alpha=\sum_{k=1}^\infty\frac{k^{-\alpha}}{k+1}$. The inequality
\[
\sup_n\Big|\sum_{k=n}^\infty\frac{x_k}k\Big|n^\alpha\le A\sup_n\Big|\frac1n\sum_{k=1}^nx_k\Big|n^\alpha
\]
holds for all $x\in\mathcal D(C^*)$ with 
\[
A=\begin{cases}\infty,&\alpha\le0;\\
1+\dfrac1\alpha,&0<\alpha\le1;\\
2^\alpha M_\alpha,&\alpha>1.\end{cases}
\]
It holds for all nonnegative $x\in\mathcal D(C^*)$ with
\[
A=\begin{cases}\infty,&\alpha\le0;\\
\dfrac1\alpha,&0<\alpha\le1;\\
0,&\alpha>1.\end{cases}
\]
In each case the constant $A$ is best possible.
\end{theorem}
\begin{proof}Take $u_k=k^{-\alpha}$ and $v_n=n^\alpha$ in Theorem \ref{BestConstants} to see that the best constant $A$, taken over all $x\in\mathcal D(C^*)$, is
\[
A(C^*,C)=\max\Big(M_\alpha,\sup_{n\ge2}\Big(\Big(\frac{n-1}n\Big)^{1-\alpha}+n^\alpha\sum_{k=n}^\infty\frac{k^{-\alpha}}{k+1}\Big)\Big)
\]
and the best constant $A$, taken over all nonnegative $x\in\mathcal D(C^*)$, is
\[
A^+(C^*,C)=\sup_nn^\alpha\sum_{k=n}^\infty\frac1{k(k+1)}\inf_{j\ge k}j^{1-\alpha}.
\]

We will need another monotonicity result in the spirit of Remark 4.10 of \cite{B}: If $0<\alpha\le1$, then
\[
n^\alpha\sum_{k=n}^\infty \frac{k^{-\alpha}}{k+1}
\] 
increases with $n$ to $1/\alpha$ and if $\alpha>1$, it decreases with $n$. For all $x>0$, the derivative of $\log((x+1)x^\alpha(x^{-\alpha}-(x+1)^{-\alpha}))$ is 
\[
\frac{\big(1+\frac1x\big)^\alpha-\big(1+\frac\alpha x\big)}{(x+1)^{\alpha+1}(x^{-\alpha}-(x+1)^{-\alpha})}
\]
which is nonpositive when $\alpha\le1$ and nonnegative when $\alpha\ge1$ by Bernoulli's inequality. Thus,
\[
a_k=\frac1{(k+1)k^\alpha(k^{-\alpha}-(k+1)^{-\alpha})}
\]
is a nondecreasing sequence when $\alpha\le1$ and is a decreasing sequence when $\alpha\ge1$. Its moving average,
\[
\frac{\sum_{k=n}^\infty a_k(k^{-\alpha}-(k+1)^{-\alpha})}{\sum_{k=n}^\infty (k^{-\alpha}-(k+1)^{-\alpha})}=n^\alpha\sum_{k=n}^\infty\frac{k^{-\alpha}}{k+1}.
\]
shares its monotonicity. If $0<\alpha\le1$, then the last expression goes to $1/\alpha$, as can be seen from the estimates
\[
n^\alpha\int_{n+1}^\infty (x+1)^{-\alpha-1}\,dx\le 
n^\alpha\sum_{k=n}^\infty\frac{k^{-\alpha}}{k+1}\le 
n^\alpha\int_{n-1}^\infty x^{-\alpha-1}\,dx.
\]

If $\alpha\le0$, $A(C^*,C)=\infty$ since the sums diverge. If $0<\alpha\le1$, then $A(C^*,C)=1+1/\alpha$. If $\alpha>1$, then
\[
A(C^*,C)=\max(M_\alpha, 2^{\alpha-1}+2^\alpha(M_\alpha-1/2))
=2^\alpha M_\alpha.
\]

If $\alpha>1$, then $\inf_{j\ge k}j^{1-\alpha}=0$ so $A^+(C^*,C)=0$. If $\alpha\le1$, then $\inf_{j\ge k}j^{1-\alpha}=k^{1-\alpha}$ so
$A^+(C,C^*)=\sup_nn^\alpha\sum_{k=1}^\infty\frac {k^{-\alpha}}{k+1}$, which is infinite when $\alpha\le0$ and equals $1/\alpha$ when $0<\alpha\le1$.

\end{proof}

\subsection*{Acknowledgment} 
The second author would like to thank International Science Program (ISP) for its financial support. The third author was supported by the Natural Sciences and Engineering Research Council of Canada.

\end{document}